\RequirePackage[l2tabu,orthodox]{nag}  
\documentclass[letterpaper,11pt]{article}

\usepackage{url}
\usepackage{hyperref}

\usepackage{amsmath,amsthm}
\usepackage{amsfonts}
\usepackage{latexsym}
\usepackage{amssymb,bm}
\usepackage{mathrsfs}
\usepackage{mathtools}
\usepackage[page]{appendix}

\usepackage{changes}
\usepackage{color}

\usepackage{enumitem}

\newtheorem{theorem}{Theorem}[section]
\newtheorem{lemma}[theorem]{Lemma}
\newtheorem*{lemma*}{Lemma}
\newtheorem{corollary}[theorem]{Corollary}

\newcommand{\inner}[2]{\langle #1,#2\rangle}

\newcommand{\norm}[1]{\|{#1}\|}
\newcommand{\wto}{{\;\stackrel {w\,\,} \to\; }}

\newcommand\set[1]{\{#1\}}

\newcommand{\R}{\mathbb{R}}
\newcommand{\tos}{\rightrightarrows}

\newcommand{\comenta}[1]{}

\title{A simplified proof of weak convergence in Douglas-Rachford method
to a solution of the unnderlying inclusion problem}

\author{B. F. Svaiter\thanks{IMPA, Estrada Dona Castorina 110,
    22460--320 Rio de Janeiro, Brazil ({\texttt{benar@impa.br}}) 
    tel: 55 (21) 25295112, fax: 55 (21)25124115.  }\hspace{.5em}
  \thanks{Partially supported by CNPq
   grant 306247/2015-1
    and by 
    FAPERJ
    grant Cientistas de Nosso Estado
    E-26/201.584/2014
         and
    E-26/203.318/2017}}

\begin{document}
\maketitle
\begin{abstract}
  Douglas-Rachford method is a splitting algorithm for finding a zero
  of the sum of two maximal monotone operators.
  Weak convergence
  in this method
  to
  a solution of the underlying monotone inclusion problem
  in the general case remained an open problem for 30 years and was
  prove by the author 7 year ago.
  The proof presented at that occasion was cluttered with technicalities
  because we considered the inexact version with summable errors.
  The aim of this note is to present a streamlined proof of this
  result.
  \\
  \\
  2000 Mathematics Subject Classification: 47H05, 49M27,
  49J52,
  49J45.
  \\
  \\
  Key words: Douglas-Rachford method, weak convergence, 
  monotone operators.
\end{abstract}

\pagestyle{plain}

Douglas-Rachford method is an iterative splitting algorithm for
finding a zero of a sum of two maximal monotone operators.  It was
originally proposed by Douglas and Rachford for solving the
discretized Poisson equation
\begin{equation*}
  A(w)+B(w)=f\qquad w\in W,
\end{equation*}
where $A$ and $B$ are, respectively, the discretization of
$-\partial^2/\partial x^2$ and $-\partial^2/\partial y^2$ in the
interior of the domain and $f$ is the problem's datum, while $W$ is the
family of functions on the discretized domain which satisfies the
prescribed boundary condition,
With this notation, the method proposed
by Douglas-Rachford~\cite[Part II, eq.\ (7.4)]{MR0084194} writes
\begin{align*}
  \lambda_n A(y_{n+1/2})+y_{n+1/2}
  &=x_n-\lambda_nB(x_n)\\
  \lambda_n B(x_{n+1})+x_{n+1}
  &=x_n-\lambda_nA(y_{n+1/2}),
\end{align*}
with $f=0$ or with $f$ incorporated to $A$.

Lions and Mercier~\cite{MR0551319} extended this procedure with
$\lambda=\lambda_n$ fixed, for arbitrary maximal monotone operators,
although in their analysis, for the general case,
weak convergence of
an associated sequence, but
not of $\set{x_n}$, was proved.
The solution was to be retrieved applying the resolvent of $B$ to the
weak limit of the associated sequence.

Weak convergence of $\set{x_n}$ to a solution was proved
by the author~\cite{MR2783211}, 30 years afterwards Lions and Mercier seminal
work.
This proof of weak convergence was cluttered with
technicalities because the inexact case with summable error
was considered.
The aim of this note is to present a streamlined version of that proof.
This note does not present any new result or idea, we just
developed more directly the ideas of~\cite{MR2783211} without the
technicalities that attends the use of summable error criteria.

\section{Basic definitions and results}
\label{sec:bdr}

From now on, $X$ is a real Hilbert space with inner product
$\inner{\cdot}{\cdot}$ and associated norm $\norm{\cdot}$.
Strong and weak convergence of a sequence $\set{x_n}$ in $X$ to $x$
will be denoted by $x_n \to x$ and $x_n \wto x$, respectively.
We consider in $X\times X$ the canonical inner product and norm of
Hilbert space products,
\begin{align*}
  \inner{(z,w)}{(z',w')}
  =\inner{z}{z'}+\inner{w}{w'},\qquad
  \norm{(z,w)}=\sqrt{\inner{(z,w)}{(z,w)}},
\end{align*}
which makes it also an Hilbert space.

A point to set operator $T:X\tos X$ is a relation in $X$, that is
$T\subset X\times X$; for any $x\in X$, 
\begin{align*}
  T(x)=\set{y:(x,y)\in T},\qquad
  T^{-1}(y)=\set{x: (x,y)\in T}.
\end{align*}
For $ S:X\tos $, $T:X\tos X $, and $ \lambda \in \R $,
the operators 
$ S+T \tos X$ and $\lambda T : X \tos X $
are defined, respectively, as
\begin{align*}
  (S+T)(x)=\set{y+y' : y\in S(x),y'\in T(x)},\qquad
  (\lambda T)(x)=\set{\lambda y : y\in T(x)}.
\end{align*}
A function $f:X\supset D\to X$ is identified with the point to set operator
$F=\set{(x,y): x\in D,\;y=f(x)}$.

An operator $ T : X \tos X $ is \emph{monotone} if
\begin{align*}
  \inner{x-x'}{y-y'} \geq 0 \qquad
  \forall\, (x,y),(x',y')\in T
\end{align*}
and it is \emph{maximal monotone} if it is a maximal element in the
family of monotone operators in $X$ with respect with the partial
order of inclusion, when regarded as a subset of $ X \times X $.

\emph{Minty's Theorem}~\cite{MR0169064} states, among other things,
that if $T$ is maximal monotone, then for each $z\in Z$ there is a
unique $(x,y)\in X\times X$ such that
\begin{align*}
  x+y=z,\qquad y\in T(x).
\end{align*}
The \emph{resolvent}, \emph{proximal mapping}, or \emph{proximal operator}
of a maximal monotone operator $T$ with stepsize $\lambda>0$ is
$J_{\lambda T}=(I+\lambda T)^{-1}$.
In view of Minty's theorem, the proximal mapping (of a maximal monotone
operator) is a function whose domain is the whole underlying Hilbert
space.

Opial's Lemma~\cite{MR0211301}, which we state next, will be used in our
proof.

\begin{lemma}[Opial's Lemma]
  \label{lm:op}
  Let $\set{u_n}$ be a sequence in a Hilbert space $Z$. If
  $u_n\wto u$, then for any $v\neq u$
  \begin{align*}
    \liminf_{n\to\infty}\norm{u_n-v}>\liminf_{n\to\infty}\norm{u_n-u}.
  \end{align*}
\end{lemma}

The following  trivial lemma is a particular case of a more
general result~\cite{MR1783979}.

\begin{lemma}
  \label{lm:prox}
  If $T:X\tos X$ is monotone, $v\in T(x)$, $v'\in T(x')$ and
  \begin{align*}
    v+x-(v'+x')=r
  \end{align*}
 then $\norm{v-v'}^2+\norm{x-x'}^2\leq\norm{r}^2$.
\end{lemma}

\begin{proof}
  Write $\norm{r}^2=\norm{v-v'}^2+\norm{x-x'}^2+2\inner{v-v'}{x-x'}$ and use
  the monotonicity of $T$.
\end{proof}

We will also need the following specialization of a result due
Bauschke~\cite[Corollary 3]{MR2556354}.

\begin{lemma}
  \label{lm:a2}
  Suppose $T_1$ and $T_2$ maximal monotone operators in $X$. If
  $(z_{k,i},v_{k,i})\in T_i$ for $i=1,2$,  $k=1,2,\dots$, and
  \[
    z_{k,1}-z_{k,2}\to 0,\quad v_{k,1}+v_{k,2}\to 0,\quad z_k\wto z,
    \text{ and }v_{k,1}\wto v
    \qquad\text{ as }k\to\infty
  \]
  then $(z,v)\in T_1$ and $(z,-v)\in T_2$.
\end{lemma}

\section{Douglas-Rachford method}
\label{sec:cs}

From now on $A,B:X\tos X$ are maximal monotone operators. We are
concerned with the problem
\begin{align}
  \label{eq:msip}
  0\in A(x)+B(x).
\end{align}
Lions and Mercier's extension of Douglas-Rachford method can be
stated as:
choose $\lambda>0$, $x_0,b_0\in X$ and for $n=1,2,\dots$
\begin{align}
  \label{eq:dri}
  \begin{alignedat}{4}
    &\text{compute }y_n,a_n\text{ such that}\;\;
    & a_n&\in A(y_n),\;\;
    &\lambda a_n+y_n
    &=x_{n-1}-\lambda b_{n-1},
    \\
    &\text{compute }x_n,b_n\text{ such that}
    &\;\; b_{n}&\in B(x_{n}),
    &\lambda b_{n}+x_{n}
    &=y_n+\lambda b_{n-1}.
  \end{alignedat}
\end{align}
Well definedness of this procedure follows from Minty's Theorem.
\emph{Previous convergence result~\cite{MR0551319} for general maximal
monotone operators where weak convergence of $\set{x_k+b_k}$ to a point in
the pre-image of the solution set of \eqref{eq:msip} by the mapping
$(I+B)_{-1}$.}
In general the resolvent is not sequentially
weakly continuous, therefore, this result does not implies weak convergence of
$\set{x_k}$ to a solution.

In what follows $\set{(x_n,b_n)}$ and $\set{(y_n,a_n)}$ is a pair of
sequences generated by Douglas-Rachford method.
Define
\begin{align}
  \zeta_n=y_n+\lambda b_{n-1}.
\end{align}
It follows from \eqref{eq:dri} that
\begin{align*}
  \zeta_n=J_{\lambda A}(2J_{\lambda B}-I)\zeta_{n-1}
  +(I-J_{\lambda B})\zeta_{n-1} \qquad(n=1,2,\dots),
\end{align*}
which is another classical presentation of Douglas-Rachford method.
Let us write some of Lions and Mercier's results on this method with the
notation~\eqref{eq:dri}.

\begin{theorem}[Lions and Mericer~\cite{MR0551319}]
  \label{th:lm}
  If~\eqref{eq:msip} has a solution, then $\set{x_n}$ is bounded,
  \begin{align*}
    \zeta_n-\zeta_{n-1}=y_n-x_{n-1}=-\lambda(a_n+b_{n-1})\to 0,\quad
    x_n-x_{n-1}=-\lambda(a_n+b_n) \to 0 \qquad
    \text{as }n\to \infty,
  \end{align*}
  and $\set{\zeta_n}$ converges weakly to a $\zeta$ such that
  $J_{\lambda B}(\zeta)$ is a solution of this inclusion
  problem.
\end{theorem}

The following corollary is an immediate consequence of those results
of Lions and Mercier work~\cite{MR0551319} summarized in
Theorem~\ref{th:lm}.

\begin{corollary}
  \label{cr:lm2}
  If~\eqref{eq:msip} has a solution, then 
  \begin{align*}
    y_n-x_n\to0,\;\;b_n-b_{n-1}\to 0,\;\;
    a_n-a_{n-1}\to0,\;\;y_n-y_{n-1}\to 0
  \end{align*}
  as $n\to\infty$.
\end{corollary}

\begin{proof}
  Write
  \begin{align*}
    &y_n-x_n=y_n-x_{n-1}-(x_n-x_{n-1}),
    \\
    &b_n-b_{n-1}=a_n+b_n-(a_n+b_{n-1}),
    \\
    &a_n-a_{n-1}=a_n+b_{n-1}-(a_{n-1}+b_{n-1}),
    \\
    &y_n-y_{n-1}=y_n-x_{n-1}-(y_{n-1}-x_{n-2})+(x_{n-1}-x_{n-2})
  \end{align*}
  and use Theorem~\ref{th:lm}.
\end{proof}

The \emph{extended solution set} of~\eqref{eq:msip}, as defined
in~\cite{MR2322888}, is
\begin{align}
  \label{eq:se}
  S_e(A,B)=\set{(z,w):w\in B(z),\;-w\in A(z)}.
\end{align}
It is trivial to verify that $z$ is a solution of \eqref{eq:msip}
if and only if $(z,w)\in S_e(A,B)$ for some $w$.
This set will be instrumental in the proof of weak convergence of
$\set{x_n}$ to a solution of~\eqref{eq:msip}.
It is easy to prove that $S_e(A,B)$ is closed and convex;
however, we will not explicitly use these properties.

Our aim is to prove the following theorem.
\begin{theorem}
  \label{th:main}
  Let $\set{a_k}$, $\set{y_k}$, $\set{b_k}$, and $\set{x_k}$
  are sequences generated by Douglas-Rachford method~\eqref{eq:dri}.
  If the solution set of \eqref{eq:msip} is non-empty, then
  \begin{enumerate}
  \item\label{it:t2}
    $\norm{x_n-y_n}\to 0$ and $\norm{a_n+b_n}\to 0$
    as $n\to \infty$.
  \item\label{it:t3}
    $\norm{x_n-x_{n-1}}\to 0$ and $\norm{b_n-b_{n-1}}\to 0$,
    as $n\to \infty$;
  \item\label{it:t4}
    $\norm{y_n-y_{n-1}}\to 0$, and $\norm{a_n-a_{n-1}}\to 0$
    as $n\to \infty$;
  \end{enumerate}
  \item\label{it:t1}
  and the sequences  $\set{(x_n,b_n)}$, $\set{(y_n,-a_n)}$ converges weakly
    to a point in $S_e(A,B)$;
\end{theorem}

Items \ref{it:t2}, \ref{it:t3}, and \ref{it:t4} are either in
Theorem~\ref{th:lm} or in Corollary~\ref{cr:lm2}, which is a direct
consequences of Theorem~\ref{th:lm}.
Since Theorem~\ref{th:lm} summarize some results form \cite{MR0551319},
the proper reference for
items \ref{it:t2}, \ref{it:t3}, and \ref{it:t4}
of Theorem~\ref{th:main} is \cite{MR0551319}.
We do not pretend to be the authors of these items.
Nevertheless, we
will prove these items again, for the sake of completeness and because they
merge into the main result, namely, weak convergence of
$\set{(x_n,b_n)}$ and $\set{(y_k,-a_k)}$ to a solution a point in
$S_e(A,B)$.  
It follows trivially from this results that $\set{x_k}$
and $\set{y_k}$ converge weakly to a solution of \eqref{eq:msip}.

Convergence of $\set{(x_n,b_n)}$ and $\set{(y_k,-a_k)}$ to a solution
a point in $S_e(A,B)$ was proved in~\cite[Theorem 1]{MR2783211} in a
more general context, that is, in the case where the proximal
subproblems in~\eqref{eq:dri}, \eqref{eq:dris} are solved
inexactly within a summable error tolerance.
Since here we assume that there are no errors in the the solution
of the proximal subproblems in~\eqref{eq:dri}, \eqref{eq:dris} the basic ideas of the proof
of weak convergence~\cite{MR2783211} can be used without the technicalities
which attend the use of summable error criteria.
For example, instead of using Qasi-Fej\'er convergence, we will be
able to use Fej\'er convergence etc.

\section{Proof of Theorem~\ref{th:main}}
\label{sec:pmt}

Observe that $S_e(\lambda A,\lambda B)=\set{(z,\lambda w)\,:\,
    (s,w)\in S_e(A,B)}$
and for any sequence $\set{(z_n,w_n)}$ in $X \times X$,
\begin{align*}
  (z_n,w_n)\wto (z,w)\text{ as }n\to \infty\iff (z_n,\lambda w_n)\wto (z,\lambda w)\text{ as }n\to \infty.
\end{align*}
Moreover, since one can define $\tilde A=\lambda A$ and
$\tilde B=\lambda B$, and apply~\eqref{eq:dri} to $\tilde A$ and
$\tilde B$ instead of $A$ and $B$, without loss of generality we
assume from now on that $\lambda=1$.

Since we are assuming that $\lambda=1$, \eqref{eq:dri} writes
\begin{equation}
  \label{eq:dris}
  \begin{alignedat}{2}
    a_n& \in A(y_n),&\quad a_n+y_n&=x_{n-1}-b_{n-1},
    \\
    b_n& \in B(x_n),&\quad b_n+x_n&=y_n+b_{n-1},
    \qquad\qquad(n=1,2,\dots).
  \end{alignedat}
\end{equation}
Let
\begin{align}
  \label{eq:pn}
  p_n=(x_n,b_n)\qquad(n=1,2,3,\dots)
\end{align}
Our aim is to prove that $\set{p_n}$ converges weakly to a point in
$S_e(A,B)$.  First we will prove that this sequence is Fej\'er
convergent to $S_e(A,B)$.

The first inequality in the next lemma, namely Lemma~\ref{lm:fejer}, was proved
in~\cite[Lemma 2]{MR2783211}.

\begin{lemma}[\mbox{\cite[Lemma 2]{MR2783211}}]
  \label{lm:fejer}
  If $p\in S_e(A,B)$, then for all $n$,
  \begin{alignat*}{1}
    \norm{p-p_{n-1}}^2&\geq\norm{p-p_n}^2+\norm{a_n+b_{k-1}}^2,
    \\
    \norm{p-p_0}^2&\geq\norm{p-p_{n}}^2+\sum_{k=1}^n\norm{a_k+b_{k-1}}^2.
  \end{alignat*}
\end{lemma}

\begin{proof}
  Fix
  $p=(z,w)\in S_e(A,B)$.
  It follows from the inclusions in \eqref{eq:se} and~\eqref{eq:dris} and
  from the monotonicity of $A$ and $B$ 
  that
  \begin{align*}
    \inner{z-x_n}{x_n-x_{n-1}}
    &
      = \inner{z-x_n}{-a_n-b_n}
    \\
    &
      = \inner{x_n-x_n}{w-b_n}+
      \inner{x_n-y_n}{-w-a_n}
      +\inner{y_n-x_n}{-w-a_n}
    \\
    &
      \geq \inner{y_n-x_n}{-w-a_n}.
  \end{align*}
  Direct combination of this inequality with the second equality in
  \eqref{eq:dris} yields
  \begin{align*}
    \inner{p-p_n}{p_n-p_{n-1}}
    &
      = \inner{z-x_n}{x_n-x_{n-1}}+\inner{w-b_n}{b_n-b_{n-1}}
    \\
    &
      \geq \inner{y_n-x_n}{-w-a_n}+\inner{w-b_n}{y_n-x_n}. 
  \end{align*}
  Since the expression at the right hand-side of the above inequality
  does not depends on $w$, we can substitute $b_n$ for $w$ in this
  expression to conclude that
  \begin{align*}
    \inner{p-p_n}{p_n-p_{n-1}}
    \geq \inner{x_n-y_n}{a_n+b_n}.
  \end{align*}
  Therefore
  \begin{align*}
    \norm{p-p_{n-1}}^2
    &=\norm{p-p_n}^2+
      \norm{p_n-p_{n-1}}^2+2\inner{p-p_n}{p_n-p_{n-1}}^2
    \\
    & \geq \norm{p-p_n}^2+\norm{a_n+b_n}^2+
      \norm{x_n-y_n}^2+2\inner{x_n-y_n}{a_n+b_n}
  \end{align*}
  which is trivially equivalent to the first inequality of the lemma.
  The second inequality of the lemma follows trivially from the first
  one.
  \end{proof}

  \begin{proof}[proof of Theorem~\ref{th:main}]
  Suppose the solution set of \eqref{eq:msip} is nonempty. In this
  case, $S_e(A,B)$ is nonempty and it follows from
  Lemma~\ref{lm:fejer} that $\set{p_k}$ is bounded and
  \begin{align*}
    \sum_{i=1}^\infty\norm{a_i-b_{i-1}}^2 <\infty.
  \end{align*}
   Therefore
   \begin{align}
     \label{eq:lim1}
     a_n+b_{n-1}=x_{n-1}-y_n\to0\qquad\text{as }i\to \infty.
   \end{align}

   Since
   \begin{align}
     b_n+x_n-(x_{n-1}+b_{n-1})=
     y_n+b_{n-1}-(x_{n-1}+b_{n-1})=y_n-x_{n-1},
   \end{align}
   it follows from the equalities in \eqref{eq:dris}, from the inclusion $b_n\in B(x_n)$,
   $b_{n-1} \in B(x_{n-1})$,
   and from Lemma~\ref{lm:prox}
   that
   \begin{align}
     \norm{a_n+b_n}^2+\norm{x_n-y_k}^2
     &=\norm{b_n-b_{n-1}}^2+\norm{x_n-x_{n-1}}^2
   \leq
     \norm{y_n-x_{n-1}}^2=\norm{a_n+b_{n-1}}^2.
   \end{align}
   This inequality, together with \eqref{eq:lim1} imply  items \ref{it:t2} and \ref{it:t3}.
   
   To prove item~\ref{it:t4}, write
   \begin{align*}
     a_n+y_n-(a_{n-1}-y_{n-1})
     =x_{n-1}-b_{n-1}-(a_{n-1}-y_{n-1})
     =x_{n-1}-y_{n-1}-(a_{n-1}+b_{n-1})
   \end{align*}
   and use the inclusions $a_n\in A(y_n)$, $b_{n-1} \in B ( y_{n-1})$, and Lemma~\ref{lm:prox}
   to obtain the inequality
   \begin{align*}
     \norm{a_n-a_{n-1}}^2+\norm{y_n-y_{n-1}}^2
     \leq
     \norm{x_{n-1}-y_{n-1}-(a_{n-1}+b_{n-1})}^2.
   \end{align*}
   To end the proof of item \ref{it:t4} use this inequality and item~\ref{it:t2}.

   Suppose $\set{p_{n_k}}$ converges weakly to $(z,w)$. It follows from item~\ref{it:t2}
   and Lemma~\ref{lm:a2} that $(z,w)\in S_e(A,B)$. Therefore, all
   weak limit points of  $\set{p_n}$ belongs to $S_e(A,B)$.
   Since $\set{p_n}$ converges Fej\'er to  $S_e(A,B)$,
   it follows from Opial's Lemma that this sequence has at most one weak limit
   point in that set.
   Therefore, the bounded sequence $\set{p_n}$ has a unique weak limit point and such a limit
   point belongs to $S_e(A,B)$, which is equivalent to weak convergence of
   $\set{p_n}$ to a point in $S_e(A,B)$.

   Let $p\in S_e(A,B)$ be the weak limit of $\set{p_n=(x_n,b_n)}$.
   Weak convergence of $\set{(y_n,-a_n)}$ to $p$ follows trivially from
   item \ref{it:t2}.
\end{proof}

\end{document}